\documentclass{amsart}

\usepackage{amsmath}
\usepackage{amsfonts}
\usepackage{amssymb}
\usepackage{amsthm}
\usepackage{diagrams}
\usepackage{parskip}
\usepackage{verbatim}
\providecommand{\abs}[1]{\left\lvert#1\right\rvert}

\newtheorem*{theorem*}{Theorem}
\newtheorem{theorem}{Theorem}[section]
\newtheorem{proposition}[theorem]{Proposition}
\newtheorem{lemma}[theorem]{Lemma}
\newtheorem{conjecture}[theorem]{Conjecture}

\theoremstyle{definition}
\newtheorem{definition}[theorem]{Definition}

\theoremstyle{remark}

\providecommand{\Nn}{\mathbb{N}}

\providecommand{\Zz}{\mathbb{Z}}

\title{}
\author{}

\begin{document}
\title{Topological Factors of Rank-One Subshifts}
\author{Su Gao}
\address{Department of Mathematics, University of North Texas, 1155 Union Circle \#311430, Denton, TX 76203, USA}
\email{sgao@unt.edu}
\author{Caleb Ziegler}
\address{2150 Vine Street, Denver, CO 80205, USA}
\email{zieglercaleb@gmail.com}
\email{caleb.ziegler@unt.edu}

\date{\today}
\thanks{The first author acknowledges the US NSF grants DMS-1201290 and DMS-1800323 for the support of his research. Some results in this paper appeared as a part of the second author's PhD dissertation submitted to the University of North Texas in 2018.}

\begin{abstract}
We study topological factors of rank-one subshifts and prove that those factors that are themselves subshifts are either finite or isomorphic to the original rank-one subshifts. Thus, we completely characterize the subshift factors of rank-one subshifts.
\end{abstract}
\maketitle 
\thispagestyle{empty}

\section{Introduction and Definitions}
Rank-one subshifts have been studied in \cite{GH} and \cite{GZ} as topological dynamical systems. In \cite{GH} the topological isomorphism between rank-one subshifts is completely characterized in terms of their cutting and spacer parameters. In \cite{GZ} the current authors completely characterized the maximal equicontinuous factors of rank-one subshifts. In this paper we study the topological factors of rank-one subshifts which are themselves subshifts. We provide a complete characterization as follows.

\begin{theorem*}
Let $(X, T)$ be a rank-one subshift and $(Y, S)$ be a subshift. Suppose that $(Y, S)$ is a topological factor of $(X, T)$. Then either $Y$ is finite or $(X, T)$ is isomorphic to $(Y, S)$.
\end{theorem*}

We conjecture that the theorem is still true if we drop the assumption that the factor is itself a subshift. 

In this paper, by a \textit{subshift} we mean a topological dynamical system $(X, T)$ where $X$ is a closed invariant subset of some $b^\mathbb{Z}$ for a positive integer $b>1$, and $T$ is the shift map defined by $(Tx)(k)=x(k+1)$. Since the shift map is uniformly defined, we sometimes suppress mentioning the shift map and refer to the subshift as $X$.

Fixing a sequence $(q_n)_{n\geq  0}$ of integers $>1$ and a sequence $(a_{n,i})_{n\geq 0, 1\leq i<q_n}$ of non-negative integers, we define a \textit{rank-one sequence} $(v_n)_{n\geq 0}$ of binary words by setting $v_0=0$ and 
\begin{equation}\label{eqn:vn} v_{n+1}=v_n1^{a_{n,1}}v_n\cdots v_n1^{a_{n, q_n-1}}v_n. 
\end{equation}
Note that each $v_n$ is a word that starts and ends with $0$, and each $v_n$ is an initial segment of $v_{n+1}$. This allows us to define a \textit{infinite rank-one word} $V=\lim_{n\to\infty} v_n$ and a \textit{rank-one subshift} 
$$ X_V=\{ x\in 2^\mathbb{Z}\,:\, \mbox{every finite subword of $x$ is a subword of $V$}\}. $$
The sequence $(q_n)$ is called the \textit{cutting parameter} and the doubly-indexed sequence $(a_{n,i})$ is called the \textit{spacer parameter} of the rank-one subshift $(X_V, T)$.

When the spacer parameter of a rank-one subshift $X$ is bounded, $X$ is a minimal dynamical system, that is, for all $x\in X$, the orbit $\{T^kx\,:\, k\in\mathbb{Z}\}$ is dense in $X$. When the spacer parameter is unbounded, $X$ will contain a unique fixed point $1^\mathbb{Z}$, that is, the constant 1 element. In this case, for every $x\in X-\{1^\mathbb{Z}\}$, the orbit of $x$ is dense in $X$. 

Our proof of the main theorem will be split into two cases, according to whether the spacer parameter is bounded. The proofs of the two cases will be presented in Sections \ref{bdd} and \ref{unbdd}. In the case of unbounded spacer parameter, we show a slightly stronger statement that either the factor is trivial (one-element system) or isomorphic to the original rank-one subshift.

\section{Preliminaries}
If $(v_n)$ is a rank-one sequence, then any subsequence of $(v_n)$ that starts with $v_0=0$ as the first term is also a rank-one sequence and gives rise to the same rank-one subshift. This is because, given any $m>n$, one can write $v_m$ in the format of (\ref{eqn:vn}) in terms of $v_n$ with appropriately modified cutting and spacer parameters. We call this procedure of extracting subsequences of a rank-one sequence \textit{telescoping}. It is clear that telescoping changes the cutting and spacer parameters but does not change the boundedness of the spacer parameter. 

We will fix some notation to use in the rest of the paper. If $i\leq j$ are integers we let $[i,j]$ denote the set of integers in between (and including) $i$ and $j$. For a finite word $\alpha$ we let $\mbox{lh}(\alpha)$ denote the length of $\alpha$, and think of $\alpha$ as a function with domain $[0,\mbox{lh}(\alpha)-1]$. If $0\leq i\leq j<\mbox{lh}(\alpha)$, we let $\alpha[i,j]$ denote the word $\beta$ of length $j-i+1$ where $\beta(k)=\alpha(k+i)$ for $k\in[0, j-i]$.

Let $(X, T)$ be a subshift of $2^\mathbb{Z}$ and $(Y, S)$ be a subshift of $b^{\mathbb{Z}}$ for some integer $b\geq 2$. Assume $\varphi: (X, T)\to (Y, S)$ is a factor map, that is, $\varphi$ is surjective and continuous, and for all $x\in X$,
$$ \varphi(Tx)=S\varphi(x). $$
It is well-known that, due to the compactness of $X$ and $Y$, there is a sliding block code inducing $\varphi$, that is, there exist integers $r\geq 0$ and $s$, and a partition of $2^{2r+1}$, $\{C_0, \dots, C_{b-1}\}$, such that for all $x\in X$, $k\in \mathbb{Z}$, and $j\in[0,b-1]$,  
$$ \varphi(x)(k)=j \iff x[k+s-r, k+s+r]\in C_j. $$
Intuitively, the block of $x$ in between (and including) coordinates $k+s-r$ and $k+s+r$ completely determines $\varphi(x)$ at coordinate $k$. Note that the $T^{s-r}$ is an isomorphism from $(X,T)$ to itself, and thus we may assume without loss of generality that $s=r$. Let $R=2r$. Then we have 
$$ \varphi(x)(k)=j \iff x[k, k+R]\in C_j. $$
In other words, $\varphi(x)$ at coordinate $k$ is completely determined by the block of $x$ of length $R+1$ in between coordinates $k$ and $k+R$.

\section{\label{bdd}Bounded Spacer Parameters}
Throughout this section we assume $(X, T)$ is a rank-one subshift of $2^\mathbb{Z}$ with bounded spacer parameter. In other words, $X=X_V$, where the infinite rank-one word $V$ is given by a rank-one sequence $(v_n)$, which in turn is determined by a cutting parameter $(q_n)$ and spacer parameter $(a_{n,i})$. Let $B$ be a bound for the spacer parameter. That is, for all $n\geq 0$ and $1\leq i<q_n$, we have $a_{n,i}\leq B$. 

Let $(Y, S)$ be a subshift of $b^{\mathbb{Z}}$ for some integer $b\geq 2$. Assume $\varphi$ is a factor map from $(X, T)$ onto $(Y, S)$. Let $\{C_0,\dots, C_{b-1}\}$ be the sliding block code corresponding to $\varphi$. Assume that the sliding block code has window size $R+1$, that is, for any $x\in X$, $k\in\mathbb{Z}$, and $j\in[0,b-1]$, 
$$ \varphi(x)=j\iff x[k,k+R]\in C_j. $$

By telescoping, we may assume $\mbox{lh}(v_1)\gg B+2R$. For each $n\geq 1$, let $\alpha_n$ be the block of length $\mbox{lh}(v_n)-R$ obtained from the application of the sliding block code to $v_n$, that is, for $k\in [0, \mbox{lh}(v_n)-R-1]$,
$$ \alpha_n(k)=j\iff v_n[k,k+R]\in C_j. $$
For each $a\leq B$, let $\beta_a$ be the block of length $a+R$ obtained from the application of the sliding block code to $v_1[\mbox{lh}(v_n)-R, \mbox{lh}(v_n)-1]1^av_1[0,R-1]$. Then we have that, for all $n\geq 1$,
$$ \alpha_{n+1}=\alpha_n\beta_{a_{n,1}}\alpha_n\dots \alpha_n\beta_{a_{n, q_n-1}}\alpha_n. $$

Now suppose $x\in X$. Then by Proposition 2.28 of \cite{GH}, for any $n\geq 1$, $x$ can be written uniquely as 
\begin{equation}\label{form}  \cdots\cdots v_n 1^{c_{-1}}v_n 1^{c_0} v_n 1^{c_1}v_n \cdots\cdots 
\end{equation}
where $c_i\geq 0$ for $i\in\mathbb{Z}$. Recall that these occurrences of $v_n$ are called expected occurrences of $v_n$.
It follows that $\varphi(x)$ can be written as
$$ \cdots \cdots \alpha_n \beta_{c_{-1}} \alpha_n\beta_{c_0} \alpha_n \beta_{c_1}\alpha_n\cdots\cdots$$

Since $(X, T)$ is minimal when it has bounded spacer parameter, it follows that $(Y, S)$ is also minimal. From now on we assume that $(Y, S)$ is not finite. Our objective is to show that the factor map $\varphi$ is indeed a topological isomorphism. In order to do this, it suffices to show that $\varphi$ is one-to-one, as $\varphi$ is a closed map given that $X$ and $Y$ are compact.

Toward a contradiction, we assume that there are distinct $x, x'\in X$ with $\varphi(x)=\varphi(x')$. Fix such $x, x'$. Since $x$ and $x'$ each has a unique decomposition in the form (\ref{form}), we have $k, k'\in \mathbb{Z}$ with $0<|k-k'|\leq\frac{1}{2}(\mbox{lh}(\alpha_n)+B+R)$ such that $x$ has an expected occurrence of $v_n$ at coordinate $k$ and $x'$ has an expected occurrence of $v_n$ at $k'$. To see this, note first that there must be $k\in\mathbb{Z}$ such that $x$ has an expected occurrence of $v_n$ at coordinate $k$ while $x'$ does not, since otherwise $x=x'$. Let $x$ have an expected occurrence of $v_n$ at coordinate $k$ while $x'$ does not. Suppose further that $x'$ does not have an expected occurrence of $v_n$ at $k'$ for any $k'\in [k-\frac{1}{2}(\mbox{lh}(\alpha_n)+B+R), k+\frac{1}{2}(\mbox{lh}(\alpha_n)+B+R)]$. Then for an interval of length $\mbox{lh}(\alpha_n)+R+B=\mbox{lh}(v_n)+B$ there is no expected occurrence of $v_n$. This violates (\ref{form}). 

Without loss of generality, we may assume $k<k'$. Moreover, the expected occurrence of $v_n$ in $x$ at coordinate $k$ is followed by a spacer $1^a$ and then followed by another expected occurrence of $v_n$. Similarly, the expected occurrence of $v_n$ in $x'$ at coordinate $k'$ is followed by a spacer $1^{a'}$ and then followed by another expected occurrence of $v_n$. Without loss of generality, we may assume $a\neq a'$. This is because, if $a=a'$, then instead of considering the expected occurrence of $v_n$ at $k$ in $x$ and that at $k'$ in $x'$, which we call the \textit{first} expected occurrences of $v_n$, we may consider the second expected occurrences of $v_n$ which follow the spacers specified above, and note that the difference of their beginning coordinates will still be $k'-k$. If the spacers following them are of different lengths, then we are done. Otherwise, we can repeat and consider the next expected occurrences of $v_n$ in $x$ and $x'$. By repeating, we may thus find expected occurrences of $v_n$ in $x$ and $x'$ respectively which satisfy the assumption that the spacers following them are of different lengths. If we fail to find such expected occurrences of $v_n$ to the right of the first expected occurrences of $v_n$, we may in a similar fashion search for expected occurrences that satisfy the assumption to the left of the first expected occurrences. If finally we fail to find such occurrences on both sides of the first expected occurrences, then we have that
$$ T^{k-k'}x=x' $$
and so
$$ \varphi(x)=\varphi(x')=\varphi(T^{k-k'}x)=T^{k-k'}(\varphi(x)). $$
This means that $\varphi(x)$ is periodic, and so $Y$ is finite, a contradiction.

For the rest of the proof, we fix $k<k'$ and $a\neq a'$ such that 
\begin{enumerate}
\item[(i)] $x$ has an expected occurrence of $v_n$ at $k$, followed by a spacer $1^a$, followed by another expected occurrence of $v_n$ at $k+\mbox{lh}(v_n)+a$;
\item[(ii)] $x'$ has an expected occurrence of $v_n$ at $k'$, followed by a spacer $1^{a'}$, followed by another expected occurrence of $v_n$ at $k'+\mbox{lh}(v_n)+a'$;
\item[(iii)] $k'-k\leq \frac{1}{2}(\mbox{lh}(v_n)+B)=\frac{1}{2}(\mbox{lh}(\alpha_n)+R+B)$.
\end{enumerate}

Let $y=\varphi(x)=\varphi(x')$. Then we have
\begin{enumerate}
\item[(i')] $y$ has an occurrence of $\alpha_n$ at $k$, followed by an occurrence of $\beta_a$, followed by another occurrence of $\alpha_n$ at $k+\mbox{lh}(\alpha_n)+R+a$;
\item[(ii')] $y$ has an occurrence of $\alpha_n$ at $k'$, followed by an occurrence of $\beta_{a'}$, followed by another occurrence of $\alpha_n$ at $k'+\mbox{lh}(\alpha_n)+R+a'$.
\end{enumerate}

Since $k'-k\leq \frac{1}{2}(\mbox{lh}(\alpha_n)+R+B)$, the two occurrences of $\alpha_n$ in $y$ which occur at $k$ and $k'$ overlap for at least $\frac{1}{2}(\mbox{lh}(\alpha_n)-R-B)$ coordinates. 

We use the following concept and general lemma.
\begin{definition} Let $\eta$ be a finite string and $0<p<\mbox{lh}(\eta)$. Suppose $\mbox{lh}(\eta)=lp+q$ where $0\leq q<p$. We say that $\eta$ has \textit{period} $p$ (or $p$ is a \textit{period} for $\eta$) if $(\eta[0,p-1])^l$ is an initial segment of $\eta$ and $\eta$ is an initial segment of $(\eta[0,p-1])^{l+1}$.
\end{definition}

\begin{lemma} \label{lem:offset}
Let $\eta$ be a finite word and $0<p<\mbox{lh}(\eta)$. Suppose $\eta$ occurs at both coordinates $0$ and $p$ in some sufficiently long string $\xi$. Then $p$ is a period for $\eta$.
\end{lemma}

\begin{proof}
Since $\eta$ occurs at $p$ in $\xi$, we have $\xi[p,2p-1]=\eta[0,p-1]=\xi[0,p-1]$. If $\mbox{lh}(\eta)<2p$, we have that $\eta$ is an initial segment of $\xi[0,2p-1]=(\eta[0, p-1])^2$, and so $p$ is a period of $\eta$. Otherwise, $\mbox{lh}(\eta)>2p$, we have that $\xi[0,2p-1]=(\eta[0,p-1])^2$ is an initial segment of the occurrence of $\eta$ at $0$. Considering the occurrence of $\eta$ at $p$, we have that $\xi[2p,3p-1]=\eta[p,2p-1]=\eta[0,p-1]$. It follows that $\xi[0,3p-1]=(\eta[0,p-1])^3$. Now, if $\mbox{lh}(\eta)<3p$, we have that $\eta$ is an initial segment of $\xi[0,3p-1]=(\eta[0,p-1])^3$, and so $p$ is a period of $\eta$. Otherwise, the lemma is proved by repeating this argument.
\end{proof}

Applying Lemma~\ref{lem:offset} to the occurrences of $\alpha_n$ at $k$ and $k'$, we obtain that $k'-k$ is a period of $\alpha_n$. Again, applying Lemma~\ref{lem:offset} to the occurrences of $\alpha_n$ at $k+\mbox{lh}(\alpha_n)+R+a$ and $k'+\mbox{lh}(\alpha_n)+R+a'$, we obtain that $k'-k+a'-a$ is also a period of $\alpha_n$. 

If either $0<k'-k\leq B$ or $0<k'-k+a'-a\leq B$, then we conclude that $\alpha_n$ has a period $p\leq B$. Otherwise, let $q=\min(k'-k,k'-k+a'-a)$ and $r=\max(k'-k,k'-k+a'-a)$. Then $0<r-q\leq B$. Also $r\leq \frac{1}{2}(\mbox{lh}(\alpha_n)+R+3B)$. Let $\gamma_n=\alpha_n[r,\mbox{lh}(\alpha_n)-1]$. Then $\mbox{lh}(\gamma_n)\geq \frac{1}{2}(\mbox{lh}(\alpha_n)-R-3B)$. Since $r$ is a period of $\alpha_n$, $\gamma_n$ is an initial segment of $\alpha_n$. Since $q$ is also a period of $\alpha_n$, we also have that $\gamma_n$ occurs at coordinate $q$ in $\alpha_n$. Applying Lemma~\ref{lem:offset} to the occurrences of $\gamma_n$ at coordinates $q$ and $r$ in $\alpha_n$, we obtain that $r-q$ is a period of $\gamma_n$. In other words, $r-q$ is a period of $\alpha_n[0,\mbox{lh}(\gamma_n)-1]$. In all cases we have that for some $l_n\geq \frac{1}{2}(\mbox{lh}(\alpha_n)-R-3B)$, $\alpha_n[0,l_n-1]$ has a period $p\leq B$.

By telescoping, we may assume that $l_n\geq \mbox{lh}(\alpha_{n-1})$, and it follows that $\alpha_{n-1}$ has a period $p\leq B$. Again by telescoping, we may assume that there is $p\leq B$ such that for all sufficiently large $n$, $\alpha_n$ has a period $p$. It then follows that $y$ has period $p$, and that $Y$ is finite, a contradiction.

We have thus proved 

\begin{theorem}\label{thm:bdd}
Let $(X, T)$ be a rank-one subshift with bounded spacer parameter and let $(Y,S)$ be a subshift. Suppose $(Y,S)$ is a topological factor of $(X, T)$. Then either $Y$ is finite or else $(X, T)$ is isomorphic to $(Y, S)$.
\end{theorem}

\section{\label{unbdd}Unbounded Spacer Parameters}
In this section we prove 
\begin{theorem}\label{thm:ubdd}
Let $(X, T)$ be a rank-one subshift with unbounded spacer parameter and let $(Y,S)$ be a subshift. Suppose $(Y,S)$ is a topological factor of $(X, T)$. Then either $Y$ is trivial (that is, a singleton) or else $(X, T)$ is isomorphic to $(Y, S)$.
\end{theorem}

The rest of this section is devoted to a proof of Theorem~\ref{thm:ubdd}. Throughout this section we assume $(X, T)$ is a rank-one subshift of $2^\mathbb{Z}$ with unbounded spacer parameter. We continue to use $(q_n), (a_{n,i})$, and $(v_n)$, respectively, to denote the cutting parameter, the space parameter, and the induced rank-one sequence.

We first analyze the forms of elements of $X$ as bi-infinite words. In particular, we identify all elements of $X$ with an infinite string of 1s. First, $1^\mathbb{Z}\in X$ and is a unique fixed point. Next, we recall Lemma 3.10 of \cite{GH}, which states that, given any $K\in\mathbb{Z}$, there is a unique $z\in X$ so that $z$ has a first occurrence of $0$ at coordinate $K$. Recall that $V=\lim_{n\to\infty} v_n$ is the infinite rank-one word which has each $v_n$ as its initial segment. For each $K\in \mathbb{Z}$, define
$$ z_K(k)=\left\{\begin{array}{ll}1, & \mbox{ if $k<K$,}\\V(k-K), &\mbox{ if $k\geq K$.} \end{array}\right. $$
Then $z_K$ is of the form $1^{-\mathbb{N}}V$ with the occurrence of $V$ starting at coordinate $K$. It is easy to verify that $z_K\in X$, and therefore it is the unique $z\in X$ with a first occurrence of $0$ at coordinate $K$. 

Similarly, Lemma 3.10 of \cite{GH} also gives that, for any $K\in\mathbb{Z}$, there is a unique $z\in X$ so that $z$ has a last occurrence of $0$ at coordinate $K$. To identify this element, we consider a dual infinite rank-one word $V^*$ defined in \cite{GZ}. To define $V^*$, note that each $v_n$ is also an end segment of $v_{n+1}$. This allows us to take a dual limit and obtain $V^*$, where $V^*$ has all $v_n$ as its end segment. More formally, we can define $V^*$ as an infinite word with domain $-\mathbb{N}$, where for each $k\in \mathbb{N}$, 
$$ V^*(-k)=v_n(\mbox{lh}(v_n)-k-1) $$
for any $n$ such that $\mbox{lh}(v_n)>k$. Then for any $K\in \mathbb{Z}$, define
$$ z^*_K(k)=\left\{\begin{array}{ll}1, & \mbox{ if $k>K$,}\\V^*(k-K), &\mbox{ if $k\leq K$.} \end{array}\right. $$
Then $z^*_K$ is the unique $z\in X$ with the last occurrence of $0$ at coordinate $K$. Each $z^*_K$ is of the form $V^*1^\mathbb{N}$.

In summary, the set
$$ S=\{ 1^\mathbb{Z}, z_K, z^*_K\,:\, K\in \mathbb{Z}\} $$
consists precisely of all elements of $X$ which contain an infinite string of 1s. If $x\in X-S$, then again by Proposition 2.28 of \cite{GH}, for any $n\geq 1$, $x$ can be written uniquely as
\begin{equation*}  \cdots\cdots v_n 1^{c_{-1}}v_n 1^{c_0} v_n 1^{c_1}v_n \cdots\cdots 
\end{equation*}
where $c_i\geq 0$ for $i\in\mathbb{Z}$. Once again these occurrences of $v_n$ are called expected occurrences of $v_n$. 

Let $(Y, S)$ be a subshift of $b^{\mathbb{Z}}$ for some integer $b\geq 2$. Assume that $Y$ is not a singleton. Assume $\varphi$ is a factor map from $(X, T)$ onto $(Y, S)$. We will show that $\varphi$ is a topological isomorphism. Again, it suffices to show that $\varphi$ is one-to-one.

Let $\{C_0,\dots, C_{b-1}\}$ be the sliding block code corresponding to $\varphi$. Assume that the sliding block code has window size $R+1$, that is, for any $x\in X$, $k\in\mathbb{Z}$, and $j\in[0,b-1]$, 
$$ \varphi(x)=j\iff x[k,k+R]\in C_j. $$

Applying the sliding block code to $1^\mathbb{Z}$, we obtain a constant element of $Y$ as $\varphi(1^\mathbb{Z})$. Without loss of generality, we assume $\varphi(1^\mathbb{Z})=1^\mathbb{Z}$. Thus, for any $l>R$, an application of the sliding block code to the string $1^l$ results in the string $1^{l-R}$.

By telescoping, we may assume $\mbox{lh}(v_1)\gg R$. For each $n\geq 1$, let $\alpha_n$ be the block of length $\mbox{lh}(v_n)-R$ obtained from the application of the sliding block code to $v_n$, that is, for $k\in [0, \mbox{lh}(v_n)-R-1]$,
$$ \alpha_n(k)=j\iff v_n[k,k+R]\in C_j. $$
It is clear that each $\alpha_n$ is an initial segment as well as an end segment of $\alpha_{n+1}$. We let $W$ be the infinite word taken as a limit of $\alpha_n$, that is, so that every $\alpha_n$ is an initial segment of $W$. Similarly, let $W^*$ be the dual limit of $\alpha_n$, that is, $W^*$ is an infinite word with domain $-\mathbb{N}$ so that every $\alpha_n$ is an end segment of $W^*$. Then each $\varphi(z_K)$ is of the form $1^{-\mathbb{N}}W$ and each $\varphi(z^*_K)$ is of the form $W^*1^\mathbb{N}$.

Note that for every $x\in X$, every finite subword of $x$ is a subword of $v_n$ for some $n\geq 1$. It follows that, for every $y\in Y$, every finite subword of $y$ is a subword of $\alpha_n$ for some $n\geq 1$. This implies that $\alpha_n$ cannot be constant for all $n\geq 1$, or else $Y$ would be a singleton. By telescoping, we may assume $\alpha_1$ is not constant. Without loss of generality, we may assume $\alpha_1$ contains an occurrence of $0$. 

Let $i_0$ be the first occurrence of 0 in $\alpha_1$ and $i_1$ be the last occurrence of $0$ in $\alpha_1$. Then $\varphi(z_K)$ has its first occurrence of $0$ at $K+i_0$ and $\varphi(z^*_K)$ has its last occurrence of $0$ at $K-\mbox{lh}(\alpha_1)-i_1-1$. This implies that $\varphi\upharpoonright S$ is one-to-one. For any $x\in X-S$, $\varphi(x)$ contains infinitely many $0$s in both directions. Thus $\varphi(S)\cap \varphi(X-S)=\emptyset$. To finish our proof, it remains to show that $\varphi\upharpoonright (X-S)$ is one-to-one.

For this we use the following lemma.

\begin{lemma}\label{lem:spacer} Let $L\geq \mbox{lh}(v_1)$. Let $x, x'\in X-S$ be such that for any $l\geq L$, any spacer of length $l$ in between expected occurrences of $v_1$ in $x$ occurs at the same position as a spacer of the same length in between expected occurrences of $v_1$ in $x'$. Then $x=x'$.
\end{lemma} 

\begin{proof}
First note that, since $X$ has unbounded spacer parameter, there are spacers of length $\geq L$ in between expected occurrences of $v_1$ in $x$. By telescoping, we can assume such spacers occur in $v_2$, that is, for some $1\leq j<q_1$, $a_{1,j}\geq L$. Let $j_1< \dots< j_M$ enumerate all $j\in[1,q_1-1]$ such that $a_{1,j}\geq L$. Note that $x$ can be uniquely written as a concatenation of expected occurrences of $v_2$ with spacers in between:
\begin{equation}\label{eq:xv2} \cdots\cdots v_2 1^{t_{-1}}v_2 1^{t_0} v_2 1^{t_1}v_2 \cdots\cdots 
\end{equation} 
where $t_i\geq 0$ for $i\in\mathbb{Z}$. Expanding the expression using
$$ v_2=v_11^{a_{1,1}}v_1\dots v_11^{a_{1,q_1-1}}v_1, $$
we obtain the unique expression of $x$ as a concatenation of expected occurrences of $v_1$ with spacers in between, as follows:
\begin{equation}\label{eq:x} \cdots\cdots v_1 1^{c_{-1}}v_1 1^{c_0} v_1 1^{c_1}v_1 \cdots\cdots 
\end{equation}
where $c_j\geq 0$ for all $j\in \mathbb{Z}$. Without loss of generality, assume that $1^{c_0}$ in the above expression (\ref{eq:x}) is a spacer in between expected occurrences of $v_2$ in $x$. Then we have, for all $1\leq j<q_1$, $c_j=a_{1, j}$. Furthermore, we have that for any $k\in \mathbb{Z}$ and $1\leq j<q_1$, 
$$ c_{j+kq_1}=a_{1,j}. $$
Thus, all spacers $1^{c_j}$ in the above expression (\ref{eq:x}), except for those corresponding to $j\equiv 0\  (\mbox{mod}\ q_1)$, exhibit a periodic structure with period $q_1$. In contrast, the spacers $1^{c_j}$ for $j\equiv 0\  (\mbox{mod}\ q_1)$ do not exhibit a periodic structure with period $q_1$, since their lenghs are unbounded. Among all the spacers demonstrated in (\ref{eq:x}), the ones with lengths $\geq L$ correspond to $1^{c_j}$ for all $j\equiv j_1, \dots, j_M\ (\mbox{mod}\ q_1)$ and some $j\equiv 0\  (\mbox{mod}\ q_1)$. 

Now all these observations about $x$ hold similarly for $x'$. In particular, if we similarly express $x'$ as
\begin{equation}\label{eq:x'} \cdots\cdots v_1 1^{c'_{-1}}v_1 1^{c'_0} v_1 1^{c'_1}v_1 \cdots\cdots 
\end{equation}
with $c'_j\geq 0$ for $j\in\mathbb{Z}$, there are also exactly $M$ many periodic classes of spacers $1^{c'_j}$ with $c'_j\geq L$ and one aperiodic class of spacers. 

By our assumption, any spacer of length $\geq L$ in (\ref{eq:x}) occurs at the same position as a spacer of the same length in (\ref{eq:x'}). Thus the spacers in (\ref{eq:x}) corresponding to $j\equiv j_1, \dots, j_M\ (\mbox{mod}\ q_1)$ must align with the corresponding spacers in (\ref{eq:x'}). It follows that any expected occurrence of $v_2$ in $x$ occurs at the same position as an expected occurrence of $v_2$ in $x'$. By Proposition 2.28 of \cite{GH}, we conclude that $x=x'$.
\end{proof}

We are now ready to show that $\varphi\upharpoonright (X-S)$ is one-to-one. Let $x, x'\in X-S$ and assume that $\varphi(x)=\varphi(x')$. Consider a subword of $x$ of the form $v_11^cv_1$ where $c\geq R$. An application of the sliding block code to $v_11^cv_1$ results in a word of the form
$$ \alpha_1\eta 1^{c-R}\epsilon\alpha_1 $$
for some words $\eta$ and $\epsilon$ with $\mbox{lh}(\eta)=\mbox{lh}(\epsilon)=R$. Therefore, whenever $v_11^cv_1$ occurs in $x$, there is an occurrence of
$\alpha_1\eta 1^{c-R}\epsilon\alpha_1$ in $\varphi(x)$ at the same position. 

Let $L=\mbox{lh}(v)+2R$. We now verify that the condition for Lemma~\ref{lem:spacer} holds for $x$ and $x'$. For this, suppose $l\geq L$ and $v_11^lv_1$ occurs in $x$ at position $k$, with both ocurrences of $v_1$ expected. As discussed above, there is an occurrence of $\beta=\alpha_1\eta 1^{l-R}\epsilon\alpha_1$ in $\varphi(x)$ at position $k$. Since $\varphi(x)=\varphi(x')$, we get an occurrence of $\beta$ in $\varphi(x')$ at the same position. Note that there are at least $\mbox{lh}(v_1)+R$ many consecutive 1s in $\beta$. It follows that the occurrence of $\beta$ in $\varphi(x')$ must be induced by an occurrence of $v_11^c v_1$ in $x'$ where $c\geq R$. To see this, assume $c<R$ and that an occurrence of $v_11^cv_1$ induces $\beta$. Then an application of the sliding block code gives an occurrence of $\alpha_1$ at each of the occurrences of $v_1$. Since $\alpha_1$ contains at least one occurrence of $0$, the number of consecutive 1s is bounded by $\mbox{lh}(v_1)+c<\mbox{lh}(v_1)+R$, a contradiction.

Since $c\geq R$, an application of the sliding block code yields $\alpha_1\eta 1^{c-R}\epsilon\alpha_1$. We claim $c=l$. To see this, we consider four cases depending on whether $\eta=1^R$ and whether $\epsilon=1^R$. We only argue for the case in which both $\eta\neq 1^R$ and $\epsilon\neq 1^R$, and the other cases are similar. In this case both $\eta$ and $\epsilon$ contain occurrences of symbols other than 1. It is clear that the number of consecutive 1s in $\eta 1^{l-R} \epsilon$ and in $\eta 1^{c-R}\epsilon$ would be different if $l\neq c$. 

We have seen that the occurrence of $\beta$ in $\varphi(x')$ is uniquely determined by an occurrence of $v_11^lv_1$ in $x'$. Thus, for every occurrence of a spacer of length $l\geq L$ in $x$, there is an occurrence of a spacer of the same length at the same position in $x'$. By Lemma~\ref{lem:spacer}, $x=x'$.

\end{document}